\documentclass[a4paper,12pt]{article}
\usepackage{amsmath,amssymb,amsthm}
\usepackage{amscd}

\makeatletter

\@addtoreset{equation}{section}
\makeatother

\newtheorem{dfn}{Definition}[section]
\newtheorem{prop}[dfn]{Proposition}
\newtheorem{thm}[dfn]{Theorem}
\newtheorem{lem}[dfn]{Lemma}

\newtheorem{exple}[dfn]{Example}
\newtheorem{rem}[dfn]{Remark}

\begin{document}
\title{A combinatorial construction of symplectic expansions}
\author{Yusuke Kuno}
\date{}
\maketitle

\begin{abstract}
The notion of a symplectic expansion directly relates
the topology of a surface to formal symplectic geometry.
We give a method to construct a symplectic expansion
by solving a recurrence formula given in terms of the 
Baker-Campbell-Hausdorff series.
\end{abstract}

\section{Introduction}
Let $\Sigma$ be a compact connected oriented surface of genus $g>0$
with one boundary component.
Choose a basepoint $*$ on the boundary $\partial \Sigma$ and
let $\pi=\pi_1(\Sigma,*)$ be the fundamental group of $\Sigma$.

The notion of {\it {\rm (}generalized{\rm )} Magnus expansions} was introduced
by Kawazumi \cite{Ka} in his study of the mapping class group of a surface.
By definition, the mapping class group $\mathcal{M}_{g,1}$ is
the group of homeomorphisms of $\Sigma$ fixing $\partial \Sigma$
pointwise, modulo isotopies fixing $\partial \Sigma$ pointwise.
The group $\mathcal{M}_{g,1}$ faithfully acts on $\pi$, a free group
of rank $2g$, and
it is known as {\it the theorem of Dehn-Nielsen} that $\mathcal{M}_{g,1}$
is identified with a subgroup of the automorphism group of a free group:
$$\mathcal{M}_{g,1}=\{ \varphi\in {\rm Aut}(\pi); \varphi(\zeta)=\zeta \}.$$
Here, $\zeta\in \pi$ is the element corresponding to the boundary. See \S2.
By choosing a Magnus expansion,
the completed group ring of $\pi$ (with respect to the augmentation ideal)
is identified with the completed tensor algebra generated by the
first homology of the surface. In this way we obtain a tensor
expression of the action of $\mathcal{M}_{g,1}$ on $\pi$.
From this point of view Kawazumi obtained
extensions of the Johnson homomorphisms $\tau_k$
introduced by Johnson \cite{Joh1} \cite{Joh2}.
For details, see \cite{Ka}.

Actually the treatment in \cite{Ka} is on the automorphism
group of a free group, rather than the mapping class group.
There are infinitely many Magnus expansions, and the arguments
in \cite{Ka} hold for any Magnus expansions.
Recently, Massuyeau \cite{Mas} introduced the notion of
{\it symplectic expansions}, which are Magnus expansions satisfying
a certain kind of boundary condition, which comes from the fact
that $\pi$ has a particular element corresponding to the boundary
$\partial \Sigma$. Some nice properties of symplectic expansions
are clarified by \cite{KK}. In particular, it is shown that there is
a Lie algebra homomorphism from the Goldman Lie algebra of $\Sigma$
(see Goldman \cite{Go}) to ``associative",
one of the three Lie algebras in formal symplectic geometry
by Kontsevich \cite{Kon}, via a symplectic expansion
(see \cite{KK} Theorem 1.2.1).

Although there are infinitely many symplectic expansions
(see \cite{KK} Proposition 2.8.1), there are not so many known examples.
The boundary condition is so strong to be satisfied.
For instance, {\it the fatgraph Magnus expansion}
given by Bene-Kawazumi-Penner \cite{BKP}, is unfortunately, not symplectic.
Kawazumi \cite{Ka2} \S6 first constructed an $\mathbb{R}$-valued
symplectic expansion, called {\it the harmonic Magnus expansion},
by a transcendental method. Massuyeau \cite{Mas} Proposition 5.6 also gave a
$\mathbb{Q}$-valued symplectic expansion using {\it the LMO functor}.

The purpose of this paper is to present another construction of
symplectic expansions. Our construction is elementary and
suitable for computer-aided calculation.

\begin{thm}
\label{comb.exp}
There is an algorithm to construct a symplectic
expansion $\theta^{\mathcal{S}}$ associated to any free generating
set $\mathcal{S}$ for $\pi$.
\end{thm}

It should be remarked here that in the proof of the existence of
symplectic expansions (\cite{Mas} Lemma 2.16), Massuyeau already showed
how to construct a symplectic expansion degree after degree. Our
construction is also inductive, but by using {\it the Dynkin idempotents}
it fixes the choices that had to be done in the inductive step of
\cite{Mas} Lemma 2.16, hence is canonical. Moreover, our construction
works for any free generating
set for $\pi$ whereas \cite{Mas} Lemma 2.16 only deals with symplectic generators.

In \S2, we recall Magnus expansions and symplectic expansions.
Theorem \ref{comb.exp} will be proved in \S3.
In \S4 we show a naturality of our construction under the
action of a subgroup of ${\rm Aut}(\pi)$ including the mapping class
group $\mathcal{M}_{g,1}$.
In \S5, we discuss the symplectic expansion associated
to symplectic generators.

\section{Basic notions}
We denote by $\zeta$ the loop parallel to $\partial \Sigma$ and
going by counter-clockwise manner. Explicitly, if we take symplectic
generators $\alpha_1,\beta_1,\ldots, \alpha_g,\beta_g \in \pi$ as shown
in Figure 1, $\zeta=\prod_{i=1}^g [\alpha_i,\beta_i]$.
Here our notation for commutators is $[x,y]:=xyx^{-1}y^{-1}$.

\begin{center}
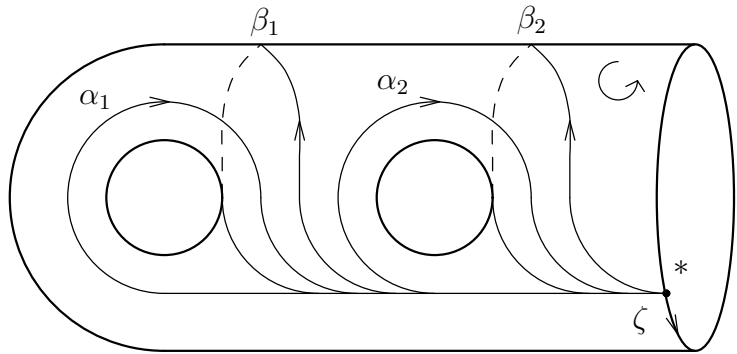
Figure 1: symplectic generators for $g=2$
%WinTpicVersion3.08
\unitlength 0.1in
\begin{picture}( 37.5000, 18.2000)(  2.0000,-18.7000)
% ELLIPSE 1 0 3 0
% 4 3750 1070 3550 1870 3550 1870 3550 1870
% 
\special{pn 13}%
\special{ar 3750 1070 200 800  0.0000000 6.2831853}%
% CIRCLE 1 0 3 0
% 4 1000 1070 1000 270 1000 270 1000 1470
% 
\special{pn 13}%
\special{ar 1000 1070 800 800  1.5707963 4.7123890}%
% CIRCLE 1 0 3 0
% 4 1000 1070 1000 770 1000 770 1000 770
% 
\special{pn 13}%
\special{ar 1000 1070 300 300  0.0000000 6.2831853}%
% CIRCLE 1 0 3 0
% 4 2400 1070 2400 770 2400 770 2400 770
% 
\special{pn 13}%
\special{ar 2400 1070 300 300  0.0000000 6.2831853}%
% CIRCLE 2 0 3 0
% 4 1000 1070 1500 1070 1500 1070 1000 1370
% 
\special{pn 8}%
\special{ar 1000 1070 500 500  1.5707963 6.2831853}%
% CIRCLE 2 0 3 0
% 4 2400 1070 2900 1070 2900 1070 2400 1570
% 
\special{pn 8}%
\special{ar 2400 1070 500 500  1.5707963 6.2831853}%
% CIRCLE 2 0 3 0
% 4 2000 1070 1500 1070 1500 1070 2000 1470
% 
\special{pn 8}%
\special{ar 2000 1070 500 500  1.5707963 3.1415927}%
% CIRCLE 2 0 3 0
% 4 3400 1070 2900 1070 2900 1070 3400 1670
% 
\special{pn 8}%
\special{ar 3400 1070 500 500  1.5707963 3.1415927}%
% CIRCLE 2 0 3 0
% 4 1800 1070 1300 1070 1300 1070 1800 1570
% 
\special{pn 8}%
\special{ar 1800 1070 500 500  1.5707963 3.1415927}%
% CIRCLE 2 0 3 0
% 4 3200 1070 2700 1070 2700 1070 3200 1570
% 
\special{pn 8}%
\special{ar 3200 1070 500 500  1.5707963 3.1415927}%
% CIRCLE 2 0 3 0
% 4 2200 1070 1700 1070 1700 1070 2200 1570
% 
\special{pn 8}%
\special{ar 2200 1070 500 500  1.5707963 3.1415927}%
% CIRCLE 2 0 3 0
% 4 3600 1070 3100 1070 3100 1070 3600 1570
% 
\special{pn 8}%
\special{ar 3600 1070 500 500  1.5707963 3.1415927}%
% SPLINE 2 1 3 0
% 7 2700 1070 2700 970 2700 870 2700 670 2800 370 2900 270 2900 270
% 
\special{pn 8}%
\special{pa 2700 1070}%
\special{pa 2700 1038}%
\special{pa 2700 1006}%
\special{pa 2700 974}%
\special{pa 2700 942}%
\special{pa 2700 910}%
\special{pa 2700 878}%
\special{pa 2700 846}%
\special{pa 2700 814}%
\special{pa 2698 782}%
\special{pa 2698 750}%
\special{pa 2698 718}%
\special{pa 2700 686}%
\special{pa 2702 654}%
\special{pa 2706 622}%
\special{pa 2710 590}%
\special{pa 2716 558}%
\special{pa 2726 526}%
\special{pa 2736 494}%
\special{pa 2746 464}%
\special{pa 2760 434}%
\special{pa 2776 406}%
\special{pa 2794 380}%
\special{pa 2812 356}%
\special{pa 2834 332}%
\special{pa 2858 310}%
\special{pa 2882 288}%
\special{pa 2900 270}%
\special{sp 0.070}%
% SPLINE 2 0 3 0
% 7 3100 1070 3100 970 3100 870 3100 670 3000 370 2900 270 2900 270
% 
\special{pn 8}%
\special{pa 3100 1070}%
\special{pa 3100 1038}%
\special{pa 3100 1006}%
\special{pa 3100 974}%
\special{pa 3100 942}%
\special{pa 3100 910}%
\special{pa 3100 878}%
\special{pa 3100 846}%
\special{pa 3102 814}%
\special{pa 3102 782}%
\special{pa 3102 750}%
\special{pa 3102 718}%
\special{pa 3102 686}%
\special{pa 3100 654}%
\special{pa 3096 622}%
\special{pa 3090 590}%
\special{pa 3084 558}%
\special{pa 3076 526}%
\special{pa 3066 494}%
\special{pa 3054 464}%
\special{pa 3040 434}%
\special{pa 3026 406}%
\special{pa 3008 380}%
\special{pa 2988 356}%
\special{pa 2966 332}%
\special{pa 2944 310}%
\special{pa 2920 288}%
\special{pa 2900 270}%
\special{sp}%
% SPLINE 2 1 3 0
% 7 1300 1070 1300 970 1300 870 1300 670 1400 370 1500 270 1500 270
% 
\special{pn 8}%
\special{pa 1300 1070}%
\special{pa 1300 1038}%
\special{pa 1300 1006}%
\special{pa 1300 974}%
\special{pa 1300 942}%
\special{pa 1300 910}%
\special{pa 1300 878}%
\special{pa 1300 846}%
\special{pa 1300 814}%
\special{pa 1298 782}%
\special{pa 1298 750}%
\special{pa 1298 718}%
\special{pa 1300 686}%
\special{pa 1302 654}%
\special{pa 1306 622}%
\special{pa 1310 590}%
\special{pa 1316 558}%
\special{pa 1326 526}%
\special{pa 1336 494}%
\special{pa 1346 464}%
\special{pa 1360 434}%
\special{pa 1376 406}%
\special{pa 1394 380}%
\special{pa 1412 356}%
\special{pa 1434 332}%
\special{pa 1458 310}%
\special{pa 1482 288}%
\special{pa 1500 270}%
\special{sp 0.070}%
% SPLINE 2 0 3 0
% 7 1700 1070 1700 970 1700 870 1700 670 1600 370 1500 270 1500 270
% 
\special{pn 8}%
\special{pa 1700 1070}%
\special{pa 1700 1038}%
\special{pa 1700 1006}%
\special{pa 1700 974}%
\special{pa 1700 942}%
\special{pa 1700 910}%
\special{pa 1700 878}%
\special{pa 1700 846}%
\special{pa 1702 814}%
\special{pa 1702 782}%
\special{pa 1702 750}%
\special{pa 1702 718}%
\special{pa 1702 686}%
\special{pa 1700 654}%
\special{pa 1696 622}%
\special{pa 1690 590}%
\special{pa 1684 558}%
\special{pa 1676 526}%
\special{pa 1666 494}%
\special{pa 1654 464}%
\special{pa 1640 434}%
\special{pa 1626 406}%
\special{pa 1608 380}%
\special{pa 1588 356}%
\special{pa 1566 332}%
\special{pa 1544 310}%
\special{pa 1520 288}%
\special{pa 1500 270}%
\special{sp}%
% LINE 2 0 3 0
% 2 1000 1570 3600 1570
% 
\special{pn 8}%
\special{pa 1000 1570}%
\special{pa 3600 1570}%
\special{fp}%
% LINE 1 0 3 0
% 2 1000 1870 3745 1870
% 
\special{pn 13}%
\special{pa 1000 1870}%
\special{pa 3746 1870}%
\special{fp}%
% LINE 1 0 3 0
% 2 1000 270 3745 270
% 
\special{pn 13}%
\special{pa 1000 270}%
\special{pa 3746 270}%
\special{fp}%
% DOT 0 0 3 0
% 2 3600 1570 3600 1570
% 
\special{pn 20}%
\special{sh 1}%
\special{ar 3600 1570 10 10 0  6.28318530717959E+0000}%
\special{sh 1}%
\special{ar 3600 1570 10 10 0  6.28318530717959E+0000}%
% LINE 2 0 3 0
% 4 1675 770 1700 670 1700 670 1725 770
% 
\special{pn 8}%
\special{pa 1676 770}%
\special{pa 1700 670}%
\special{fp}%
\special{pa 1700 670}%
\special{pa 1726 770}%
\special{fp}%
% LINE 2 0 3 0
% 4 3075 770 3100 670 3100 670 3125 770
% 
\special{pn 8}%
\special{pa 3076 770}%
\special{pa 3100 670}%
\special{fp}%
\special{pa 3100 670}%
\special{pa 3126 770}%
\special{fp}%
% LINE 2 0 3 0
% 4 925 545 1025 570 1025 570 925 595
% 
\special{pn 8}%
\special{pa 926 546}%
\special{pa 1026 570}%
\special{fp}%
\special{pa 1026 570}%
\special{pa 926 596}%
\special{fp}%
% LINE 2 0 3 0
% 4 2325 545 2425 570 2425 570 2325 595
% 
\special{pn 8}%
\special{pa 2326 546}%
\special{pa 2426 570}%
\special{fp}%
\special{pa 2426 570}%
\special{pa 2326 596}%
\special{fp}%
% LINE 2 0 3 0
% 4 3600 1700 3660 1783 3660 1783 3646 1681
% 
\special{pn 8}%
\special{pa 3600 1700}%
\special{pa 3660 1784}%
\special{fp}%
\special{pa 3660 1784}%
\special{pa 3646 1682}%
\special{fp}%
% STR 2 0 3 0
% 3 560 555 560 605 2 0
% $\alpha_1$
\put(5.6000,-6.0500){\makebox(0,0)[lb]{$\alpha_1$}}%
% STR 2 0 3 0
% 3 1445 180 1445 230 2 0
% $\beta_1$
\put(14.4500,-2.3000){\makebox(0,0)[lb]{$\beta_1$}}%
% STR 2 0 3 0
% 3 2100 480 2100 530 2 0
% $\alpha_2$
\put(21.0000,-5.3000){\makebox(0,0)[lb]{$\alpha_2$}}%
% STR 2 0 3 0
% 3 2825 170 2825 220 2 0
% $\beta_2$
\put(28.2500,-2.2000){\makebox(0,0)[lb]{$\beta_2$}}%
% STR 2 0 3 0
% 3 3635 1425 3635 1475 2 0
% $*$
\put(36.3500,-14.7500){\makebox(0,0)[lb]{$*$}}%
% STR 2 0 3 0
% 3 3425 1735 3425 1785 2 0
% $\zeta$
\put(34.2500,-17.8500){\makebox(0,0)[lb]{$\zeta$}}%
% CIRCLE 2 0 3 0
% 4 3350 460 3250 460 3350 360 3450 460
% 
\special{pn 8}%
\special{ar 3350 460 100 100  6.2831853 6.2831853}%
\special{ar 3350 460 100 100  0.0000000 4.7123890}%
% LINE 2 0 3 0
% 4 3450 460 3390 490 3450 460 3485 515
% 
\special{pn 8}%
\special{pa 3450 460}%
\special{pa 3390 490}%
\special{fp}%
\special{pa 3450 460}%
\special{pa 3486 516}%
\special{fp}%
\end{picture}%

\end{center}

Let $H_{\mathbb{Z}}:=H_1(\Sigma;\mathbb{Z})$ be
the first integral homology group of $\Sigma$. 
We denote $H:=H_{\mathbb{Z}} \otimes_{\mathbb{Z}} \mathbb{Q}$.
$H_{\mathbb{Z}}$ is naturally isomorphic to $\pi/[\pi,\pi]$,
the  abelianization of $\pi$.
With this identification in mind, we denote
$[x]:=x {\rm \ mod\ } [\pi,\pi]\in H_{\mathbb{Z}}$, or
$[x]:=(x {\rm \ mod\ } [\pi,\pi]) \otimes_{\mathbb{Z}} 1 \in H$,
for $x\in \pi$.

Let $\widehat{T}$ be the completed tensor algebra generated by $H$.
Namely $\widehat{T}=\prod_{m=0}^{\infty} H^{\otimes m}$, where
$H^{\otimes m}$ is the tensor space of degree $m$. For each
$p\ge 1$, denote $\widehat{T}_p:=\prod_{m\ge p}^{\infty} H^{\otimes m}$.
Note that the subset $1+\widehat{T}_1$ constitutes a subgroup
of the multiplicative group of the algebra $\widehat{T}$.

\begin{dfn}[Kawazumi \cite{Ka}]
A map $\theta\colon \pi \to 1+\widehat{T}_1$ is called a
{\rm (}$\mathbb{Q}$-valued{\rm )} Magnus expansion if
\begin{enumerate}
\item[{\rm (1)}] $\theta\colon \pi \to 1+\widehat{T}_1$ is
a group homomorphism, and
\item[{\rm (2)}] $\theta(x)\equiv 1+[x]\ {\rm mod}\ \widehat{T}_2$,
for any $x\in \pi$.
\end{enumerate}
\end{dfn}

The {\it standard} Magnus expansion defined by $\theta(s_i)=1+[s_i]$,
for some free generating set $\{s_i \}_i$ for $\pi$, is the simplest
example of a Magnus expansion. This is introduced by Magnus \cite{Mag}
and is often used in combinatorial group theory.

Let $\widehat{\mathcal{L}}\subset \widehat{T}$ be the completed free Lie algebra
generated by $H$. The bracket is given by $[u,v]:=u\otimes v-v\otimes u$, and its degree $p$-part
$\mathcal{L}_p=\widehat{\mathcal{L}}\cap H^{\otimes p}$ is successively
given by $\mathcal{L}_1=H$, and $\mathcal{L}_p=[H,\mathcal{L}_{p-1}]$, $p\ge 2$.
Via the intersection form $(\ \cdot \ )\colon H\times H \to \mathbb{Q}$
on $\Sigma$, $H$ and its dual $H^*={\rm Hom}_{\mathbb{Q}}(H,\mathbb{Q})$
are canonically identified by the map $H\cong H^*$, $X\mapsto (Y\mapsto (Y\cdot X))$.
Let $\omega\in \mathcal{L}_2 \subset H^{\otimes 2}$ be the symplectic form,
namely the tensor corresponding to $-1_H\in {\rm Hom}_{\mathbb{Q}}(H,H)
=H^* \otimes H=H\otimes H$. Explicitly,
if we take symplectic generators as in Figure 1, then
$A_i=[\alpha_i]$ and $B_i=[\beta_i]$ satisfy $(A_i\cdot B_j)=-(B_j\cdot A_i)=\delta_{ij}$
and $(A_i\cdot A_j)=(B_i\cdot B_j)=0$, hence we have
\begin{equation}
\label{eq:2-1}
\omega=\sum_{i=1}^g A_i\otimes B_i-B_i\otimes A_i=\sum_{i=1}^g [A_i,B_i].
\end{equation}

For a Magnus expansion $\theta$, let $\ell^{\theta}:=\log \theta$.
Here, $\log$ is the formal power series
$$\log (x)=\sum_{n=1}^{\infty}\frac{(-1)^{n-1}}{n}(x-1)^n$$
defined on the set $1+\widehat{T}_1$. The inverse of $\log$ is given
by the exponential
$\exp (x)=\sum_{n=0}^{\infty}(1/n!)x^n$.
Note that the Baker-Campbell-Hausdorff formula
\begin{eqnarray}
u\star v:=\log (\exp (u) \exp (v)) &=& u+v+\frac{1}{2}[u,v]
+\frac{1}{12}[u-v,[u,v]] \nonumber \\
 & & -\frac{1}{24}[u,[v,[u,v]]]+\cdots \label{eq:2-2}
\end{eqnarray}
endows the underlying set of $\widehat{\mathcal{L}}$ with a group structure.
A priori, $\ell^{\theta}$ is a map
from $\pi$ to $\widehat{T}_1$.

\begin{dfn}[Massuyeau \cite{Mas}]
\label{symp.exp}
A Magnus expansion $\theta$ is called symplectic if
\begin{enumerate}
\item[{\rm (1)}] $\theta$ is group-like, i.e.,
$\ell^{\theta}(\pi)\subset \widehat{\mathcal{L}}$, and
\item[{\rm (2)}] $\theta(\zeta)=\exp (\omega)$,
or equivalently, $\ell^{\theta}(\zeta)=\omega$.
\end{enumerate}
\end{dfn}

\begin{rem}
{\rm
Let $I\pi$ be the augmentation ideal of the group ring $\mathbb{Q}\pi$,
and $\widehat{\mathbb{Q}\pi}:=\varprojlim_m \mathbb{Q}\pi/I\pi^m$
the completed group ring of $\pi$. Any Magnus expansion $\theta$
induces an isomorphism $\theta\colon \widehat{\mathbb{Q}\pi}
\stackrel{\cong}{\to} \widehat{T}$ of complete augmented algebras.
See \cite{Ka} Theorem 1.3. Moreover, let $\langle \zeta \rangle$ be the cyclic
subgroup of $\pi$ generated by $\zeta$, and $\mathbb{Q}[[\omega]]$
the ring of formal power series in the symplectic form $\omega$,
which is regarded as a subalgebra of $\widehat{T}$ in an obvious way.
Then any symplectic expansion $\theta$ induces an isomorphism
$\theta\colon (\widehat{\mathbb{Q}\pi},\widehat{\mathbb{Q}\langle \zeta \rangle})
\to (\widehat{T},\mathbb{Q}[[\omega]])$ of complete Hopf algebras.
See \cite{KK} \S6.2.
}
\end{rem}

\section{Main construction}
We fix a free generating set
$\mathcal{S}=\{ s_1,\ldots, s_{2g} \}$ for $\pi$.
We denote $S_i:=[s_i]\in H$, $1\le i\le 2g$.
Let $x_1x_2\cdots x_p$ be the unique reduced word in $\mathcal{S}$ representing $\zeta$.

\begin{dfn}
Fix an integer $n\ge 1$. A set $\{ \ell_j(s_i); 1\le i\le 2g, 1\le j\le n \}
\subset \widehat{\mathcal{L}}$ is called a partial symplectic expansion up
to degree $n$, if
\begin{enumerate}
\item[{\rm (1)}] $\ell_1(s_i)=S_i$, for $1\le i\le 2g$,
\item[{\rm (2)}] $\ell_j(s_i)\in \mathcal{L}_j$, for $1\le i\le 2g$,
$1\le j\le n$, and
\item[{\rm (3)}] if we set $\bar{\ell}_n(s_i)=\sum_{j=1}^n \ell_j(s_i)$
for $1\le i\le 2g$, then
\begin{equation}
\label{eq:3-1}
\bar{\ell}_n(x_1)\star \bar{\ell}_n(x_2)\star \cdots \star \bar{\ell}_n(x_p)
\equiv \omega {\rm \ mod\ } \widehat{T}_{n+2}.
\end{equation}
\end{enumerate}
Here, we understand $\bar{\ell}_n(s_i^{-1})=-\bar{\ell}_n(s_i)$.
\end{dfn}

This notion could be thought of as an approximation to
a symplectic expansion. In this section we give a method to
refine an approximation up to degree $n-1$, to the one up to
degree $n$. Repeating this process, we will
obtain a symplectic expansion.

We need two lemmas.

\begin{lem}
\label{basis}
Suppose $4g$ elements $Y_1,\ldots,Y_{2g},Z_1,\ldots,Z_{2g}\in H$ satisfy
$\sum_{i=1}^{2g} Y_i\otimes Z_i=\omega \in H^{\otimes 2}$.
Then $Z_1,\ldots,Z_{2g}$ constitute a basis for $H$.
\end{lem}

\begin{proof}
Since $\omega$ corresponds to $-1_H\in {\rm Hom}_{\mathbb{Q}}(H,H)$
(see \S2), for any $X\in H$, we have
$$X=\omega(-X)=\sum_{i=1}^{2g}(-X\cdot Y_i)Z_i.$$
This shows that the $2g$ elements $Z_1,\ldots,Z_{2g}$ generate $H$. This proves
the lemma.
\end{proof}

Since $\pi$ is free, the quotient
$[\pi,\pi]/[\pi,[\pi,\pi]]$ is naturally isomorphic to $\Lambda^2H_{\mathbb{Z}}$,
the second exterior product of $H_{\mathbb{Z}}$. The isomorphism is induced by
the homomorphism $f\colon [\pi,\pi] \to \Lambda^2H_{\mathbb{Z}}$ which
maps the commutator $[x,y]$ to $[x]\wedge [y]$. Note that
$\Lambda^2 H_{\mathbb{Z}}$ is naturally identified with a subgroup of $H^{\otimes 2}$ by
$$\Lambda^2 H_{\mathbb{Z}}\to H^{\otimes 2},\ X\wedge Y\mapsto X\otimes Y-Y\otimes X,$$
and under this identification, we have $f(\zeta)=\omega$.

\begin{lem}
\label{commutator}
Let $y_1\cdots y_q$ be a word in $\mathcal{S}$ and suppose
$y_1\cdots y_q$ lies in the commutator subgroup $[\pi,\pi]$.
Then
$$f(y_1\cdots y_q)=\frac{1}{2}\sum_{i<j} [y_i]\wedge [y_j].$$
\end{lem}

\begin{proof}
We may assume $q\ge 2$. We prove the lemma by induction on $q$.
The case $q=2$ is trivial. Suppose $q>2$. Then there must exist
$i\ge 1$ such that $y_{i+1}=y_1^{-1}$, and
$$y_1\cdots y_q=y_1y_2\cdots y_iy_1^{-1}y_{i+2}\cdots y_q
=[y_1,y_2\cdots y_i]y_2\cdots y_iy_{i+2}\cdots y_q.$$
Hence
$f(y_1\cdots y_q)=f([y_1,y_2\cdots y_i])+
f(y_2\cdots y_iy_{i+2}\cdots y_q)$.
The first term equals
$$[y_1]\wedge ([y_2]+\cdots +[y_i])
=\frac{1}{2}\left( [y_1]\wedge ([y_2]+\cdots +[y_i])+
([y_2]+\cdots +[y_i])\wedge [y_{i+1}]\right)$$ since
$[y_1]=-[y_{i+1}]$, and the second term equals
$$\frac{1}{2}\sum_{\substack{k<\ell; \\ k,\ell \neq 1, i+1}}
[y_k]\wedge [y_{\ell}],$$
by the inductive assumption. This proves the lemma.
\end{proof}

Let $\Phi\colon \widehat{T}_1 \to \widehat{\mathcal{L}}$ be
the linear map defined by
$\Phi(Y_1\otimes \cdots \otimes Y_m)=[Y_1,[\cdots [Y_{m-1},Y_m]\cdots ]]$, $Y_i\in H$, $m\ge 1$.
We have $\Phi(u)=mu$ and $\Phi(uv)=[u,\Phi(v)]$ for any $u\in \mathcal{L}_m$,
$v\in \widehat{T}_1$. See Serre \cite{Ser} Part I, Theorem 8.1, p.28.
The maps $(1/m)\Phi|_{H^{\otimes m}}$ are called {\it the Dynkin idempotents}.
From these two properties we see that the restriction of the map
\begin{equation}
\label{eq:3-2}
\frac{1}{m+1}({\rm id}\otimes \Phi)\colon H^{\otimes m+1}\to H\otimes \mathcal{L}_m
\end{equation}
to $\mathcal{L}_{m+1}$ gives a right inverse of
the bracket $[\ ,\ ]\colon H\otimes \mathcal{L}_m
\twoheadrightarrow \mathcal{L}_{m+1}$.

Let $n\ge 2$ and let $\{ \ell_j(s_i); 1\le j\le n-1,\ 1\le i\le 2g \}$
be a partial symplectic expansion up to degree $n-1$.
We have
\begin{equation}
\label{eq:3-3}
\bar{\ell}_{n-1}(x_1)\star \bar{\ell}_{n-1}(x_2) \cdots
\star \bar{\ell}_{n-1}(x_p) \equiv \omega
{\rm \ mod\ } \widehat{T}_{n+1}.
\end{equation}
Let $V_{n+1}\in \mathcal{L}_{n+1}$ be the degree
$(n+1)$-part of
$\bar{\ell}_{n-1}(x_1)\star \bar{\ell}_{n-1}(x_2) \cdots
\star \bar{\ell}_{n-1}(x_p)$.
By Lemma \ref{commutator} we have
$\omega=f(\zeta)=f(x_1\cdots x_p)=\frac{1}{2}\sum_{i<j}X_i \wedge X_j=
\frac{1}{2}\sum_{i<j}(X_i\otimes X_j-X_j\otimes X_i)$,
where $X_i=[x_i]$.
Since $S_1,\ldots,S_{2g}$ constitute a basis for $H$,
we can uniquely write
\begin{equation}
\label{eq:3-4}
\omega=\frac{1}{2}\sum_{i<j}(X_i\otimes X_j-X_j\otimes X_i)=\sum_{i=1}^{2g}S_i\otimes Z_i,
\quad {\rm where\ }Z_i=\sum_{k}c_{ik}S_k,\quad c_{ik}\in \mathbb{Z}.
\end{equation}
Also, in view of applying (\ref{eq:3-2}) we write $V_{n+1}\in \mathcal{L}_{n+1}
\subset H^{\otimes n+1}$ as
$$V_{n+1}=\sum_{i=1}^{2g}S_i \otimes V_n^{S_i},\quad V_n^{S_i} \in H^{\otimes n}.$$
Now by Lemma \ref{basis}, $Z_1,\ldots, Z_{2g}$ constitute a basis for $H$,
hence the matrix $\{ c_{ik} \}_{i,k}$ is of full rank. Let $\{ d_{ik} \}_{i,k}$
be the inverse matrix of $\{ c_{ik} \}_{i,k}$.

\begin{prop}
\label{main.construction}
Notations are as above. Set $W_i:=(-1/(n+1))\Phi(V_n^{S_i})\in \mathcal{L}_n$
for $1\le i\le 2g$, and $\ell_n(s_i):=\sum_{k} d_{ik}W_k$ for $1\le i\le 2g$. Then
$\{ \ell_j(s_i); 1\le j\le n-1,\ 1\le i\le 2g \} \cup \{ \ell_n(s_i); 1\le i\le 2g \}$
is a partial symplectic expansion up to degree $n$.
\end{prop}

\begin{proof}
Set $\bar{\ell}_n(s_i)= \bar{\ell}_{n-1}(s_i)+\ell_n(s_i)$.
Understanding $\ell_n(s_i^{-1})=-\ell_n(s_i)$, we have
$\sum_{i=1}^p \ell_n(x_i)=0$ since $\zeta \in [\pi,\pi]$.
Hence we have $\bar{\ell}_n(x_1)\star \bar{\ell}_n(x_2) \cdots
\star \bar{\ell}_n(x_p) \equiv \omega
{\rm \ mod\ } \widehat{T}_{n+1}$ from (\ref{eq:3-3}).
By (\ref{eq:2-2}) we see that the degree $(n+1)$-part
of $\bar{\ell}_n(x_1)\star \bar{\ell}_n(x_2) \cdots
\star \bar{\ell}_n(x_p)$ equals
\begin{equation}
\label{eq:3-5}
V_{n+1}+\frac{1}{2}\sum_{i<j}
([X_i,\ell_n(x_j)]+[\ell_n(x_i),X_j]).
\end{equation}
Let $\lambda\colon H\rightarrow \mathcal{L}_n$ be the linear
map defined by $\lambda(S_i)=\ell_n(s_i)$, and we apply
the linear map
$[{\rm id},\lambda]\colon H^{\otimes 2}\rightarrow H^{\otimes n+1}$ to (\ref{eq:3-4}).
Then we obtain
$$\frac{1}{2}\sum_{i<j}([X_i,\ell_n(x_j)]-[X_j,\ell_n(x_i)])=\sum_{i=1}^{2g}[S_i,W_i^{\prime}],
\quad W_i^{\prime}=\sum_k c_{ik}\ell_n(s_k).$$
But $W_i^{\prime}=\sum_k \sum_j c_{ik} d_{kj}W_j=W_i$.
Hence (\ref{eq:3-5}) is equal to
$$V_{n+1}+\sum_{i=1}^{2g} [S_i,W_i]
=V_{n+1}-\frac{1}{n+1}\sum_{i=1}^{2g}[S_i,\Phi(V_n^{S_i})]
=V_{n+1}-\frac{1}{n+1}\Phi(V_{n+1})=0,$$
since $V_{n+1}\in \mathcal{L}_{n+1}$. Therefore, we have
$\bar{\ell}_n(x_1)\star \bar{\ell}_n(x_2) \cdots
\star \bar{\ell}_n(x_p) \equiv \omega
{\rm \ mod\ } \widehat{T}_{n+2}$. This completes the proof.
\end{proof}

We can now conclude the proof of Theorem \ref{comb.exp}.
Denote $\mathcal{S}=\{ s_1,\ldots, s_{2g} \}$ and set
$\ell_1(s_i):=S_i$, $1\le i\le 2g$. By the Baker-Campbell-Hausdorff
formula (\ref{eq:2-2}) and Lemma \ref{commutator},
$\{ \ell_1(s_i)\}_{1\le i\le 2g}$ is a partial symplectic
expansion up to degree $1$. Applying Proposition \ref{main.construction},
we obtain $\{ \ell_j(s_i); 1\le i\le 2g, j\ge 1\}$ satisfying (\ref{eq:3-1})
for any $n\ge 1$. Setting $\ell^{\mathcal{S}}(s_i):=\sum_{j=1}^{\infty}\ell_j(s_i)
\in \widehat{\mathcal{L}}$ and $\theta^{\mathcal{S}}(s_i):=\exp (\ell^{\mathcal{S}}(s_i))$,
we extend $\theta^{\mathcal{S}}$ to a homomorphism from $\pi$
using the universality of the free group $\pi$.
Then $\theta^{\mathcal{S}}$ is the desired symplectic expansion.
Note that the result $\theta^{\mathcal{S}}$ does not depend on the
total ordering on the set $\mathcal{S}$.
This completes the proof of Theorem \ref{comb.exp}.

\begin{rem}
{\rm
For a group-like expansion $\theta$, we denote
$\ell^{\theta}(x)=\sum_{j=1}^{\infty}\ell_j^{\theta}(x)$,
$\ell_j^{\theta}(x)\in \mathcal{L}_j$, for $x\in \pi$.
Proposition \ref{main.construction} can be phrased shortly
as: a choice of a free generating set for $\pi$ gives a
canonical way of modifying any group-like expansion $\theta$
satisfying $\ell^{\theta}(\zeta)\equiv \omega {\rm \ mod\ } \widehat{T}_{n+1}$
for some $n\ge 2$ into a group-like expansion satisfying the same
congruence with $n+1$ replaced by $n+2$, without changing
$\ell_j^{\theta}(x)$, for $1\le j\le n-1$.
}
\end{rem}
 
\section{Naturality}
Let ${\rm Aut}(\pi)$ be the automorphism group of $\pi$.
For $\varphi\in {\rm Aut}(\pi)$, let $|\varphi|$ be the filter-preserving
algebra automorphism of $\widehat{T}$ induced by the action of $\varphi$
on the first homology $H$. If $\theta$ is a Magnus expansion,
then the composite $|\varphi|\circ \theta \circ \varphi^{-1}$ is
again a Magnus expansion.

We show a naturality of the symplectic expansion $\theta^{\mathcal{S}}$
given in Theorem \ref{comb.exp}. Note that fatgraph Magnus expansions
have similar property (see \cite{BKP} Theorem 4.2).

\begin{prop}
\label{natural}
Suppose $\varphi\in {\rm Aut}(\pi)$ satisfies
$\varphi(\zeta)=\zeta$, or $\varphi(\zeta)=\zeta^{-1}$.
Then
$$\theta^{\varphi(\mathcal{S})}=
|\varphi|\circ \theta^{\mathcal{S}}\circ \varphi^{-1}.$$
\end{prop}

\begin{proof}
Let $\mathcal{S}=\{ s_1,\ldots, s_{2g}\}$.
We shall put $\mathcal{S}$ on the upper right of the objects
$V_{n+1}$, $\ell_j$, $c_{ik}$, etc, in the proof of Proposition \ref{main.construction}
to indicate their dependence on $\mathcal{S}$.

The equality we are going to prove is equivalent to
$\ell^{\varphi(\mathcal{S})}(\varphi(s_i))=|\varphi|\ell^{\mathcal{S}}(s_i)$,
or, $\ell_n^{\varphi(\mathcal{S})}(\varphi(s_i))=|\varphi|\ell_n^{\mathcal{S}}(s_i)$
for any $n\ge 1$. We prove this by induction on $n$. Since
$\ell_1^{\mathcal{\varphi(\mathcal{S})}}(\varphi(s_i))=[\varphi(s_i)]
=|\varphi|[s_i]$, the case $n=1$ is clear. Suppose $n\ge 2$.

First we assume $\varphi(\zeta)=\zeta$. Then $\varphi(x_1)\cdots \varphi(x_p)$
is a word in $\varphi(\mathcal{S})$ representing $\zeta$, and we have $|\varphi|\omega=\omega$
since $\varphi(\zeta)=\zeta$ and the homomorphism $f\colon [\pi,\pi]\to \Lambda^2H_{\mathbb{Z}}$
in \S3 is ${\rm Aut}(\pi)$-equivariant.
By the inductive assumption, we have
$\bar{\ell}_{n-1}^{\varphi(\mathcal{S})}(\varphi(s_i))
=|\varphi|\bar{\ell}_{n-1}^{\mathcal{S}}(s_i)$, hence
applying $|\varphi|$ to the congruence
$\bar{\ell}_{n-1}^{\mathcal{S}}(x_1)\star \bar{\ell}_{n-1}^{\mathcal{S}}(x_2) \cdots
\star \bar{\ell}_{n-1}^{\mathcal{S}}(x_p) \equiv \omega+V_{n+1}^{\mathcal{S}}
{\rm \ mod\ } \widehat{T}_{n+2}$, we obtain
$V_{n+1}^{\varphi(\mathcal{S})}=|\varphi|V_{n+1}^{\mathcal{S}}$.
Therefore, writing
$V_{n+1}^{\varphi(\mathcal{S})}=\sum_{i=1}^{2g}(|\varphi|S_i )
\otimes V_n^{|\varphi|S_i}$, we have $V_n^{|\varphi|S_i}=|\varphi|V_n^{S_i}$.

On the other hand, applying $|\varphi|$ to (\ref{eq:3-4}), we obtain
$$\omega=\sum_{i=1}^{2g}|\varphi|S_i \otimes Z_i^{\varphi(\mathcal{S})},\quad
Z_i^{\varphi(\mathcal{S})}=\sum_k c_{ik}|\varphi|S_k.$$
This implies $c_{ik}^{\varphi(\mathcal{S})}=c_{ik}^{\mathcal{S}}$
hence $d_{ik}^{\varphi(\mathcal{S})}=d_{ik}^{\mathcal{S}}$. We conclude
$W_i^{\varphi(\mathcal{S})}=|\varphi|W_i^{\mathcal{S}}$ and
$\ell_n^{\varphi(\mathcal{S})}(\varphi(s_i))=
|\varphi |\ell_n^{\mathcal{S}}(s_i)$, as desired.

If $\varphi(\zeta)=\zeta^{-1}$, the same argument shows that
$V_{n+1}^{\varphi(\mathcal{S})}=-|\varphi|V_{n+1}^{\mathcal{S}}$
and $c_{ik}^{\varphi(\mathcal{S})}=-c_{ik}^{\mathcal{S}}$ because
in this case $\varphi(x_1)\cdots \varphi(x_p)$ is a word in $\varphi(\mathcal{S})$
representing $\zeta^{-1}$ and we have $|\varphi|\omega=-\omega$.
Hence we again obtain
$\ell_n^{\varphi(\mathcal{S})}(\varphi(s_i))=
|\varphi |\ell_n^{\mathcal{S}}(s_i)$.
This completes the induction.
\end{proof}

\section{Symplectic generators}
Let $\mathcal{S}_0=\{ \alpha_1,\beta_1,\ldots,\alpha_g,\beta_g \}$ be symplectic generators
as in \S2, and let $\theta^0=\theta^{\mathcal{S}_0}$ be the symplectic
expansion associated to $\mathcal{S}_0$, given by the algorithm of Theorem \ref{comb.exp}.
For simplicity we write
$\alpha_1,\beta_1,\ldots,\alpha_g,\beta_g=\xi_1,\ldots,\xi_{2g}$.
Let $T\in {\rm Aut}(\pi)$ be the automorphism defined by
$T(\xi_i)=\xi_{2g+1-i}$, $1\le i\le 2g$. Then we have $T(\zeta)=\zeta^{-1}$
and $T(\mathcal{S}^0)=\mathcal{S}^0$. By Proposition \ref{natural},
we obtain a certain kind of symmetry for $\theta^0$.

\begin{prop}
Let $\theta^0$ be the symplectic expansion as above. Then
$$\theta^0(\xi_{2g+1-i})=|T|\theta^0(\xi_i),\quad 1\le i\le 2g.$$
\end{prop}

Finally, we give a more explicit formula for $\ell^{\mathcal{S}_0}$
in a form suitable for computer-aided calculation.
First we give another description of $V_{n+1}$ which does not
involve the Baker-Campbell-Hausdorff series.
Let $n\ge 2$ and let $\{ \ell_j(s_i); 1\le j\le n-1,\ 1\le i\le 2g \}$
be a partial symplectic expansion up to degree $n-1$. Set
$\bar{\theta}_{n-1}(s_i):=\exp(\bar{\ell}_{n-1}(s_i))$, and
$\bar{\theta}_{n-1}(s_i^{-1}):=\exp(-\bar{\ell}_{n-1}(s_i))$.
From (\ref{eq:3-3}), we have
$\bar{\ell}_{n-1}(x_1)\star \bar{\ell}_{n-1}(x_2) \cdots
\star \bar{\ell}_{n-1}(x_p) \equiv \omega+V_{n+1}
{\rm \ mod\ } \widehat{T}_{n+2}$. Applying the exponential,
we obtain $\bar{\theta}_{n-1}(x_1)\bar{\theta}_{n-1}(x_2)
\cdots \bar{\theta}_{n-1}(x_p)\equiv \exp(\omega)+V_{n+1}
{\rm \ mod\ } \widehat{T}_{n+2}$. Hence
\begin{equation}
\label{eq:5-1}
V_{n+1}=\left(\bar{\theta}_{n-1}(x_1)\bar{\theta}_{n-1}(x_2)
\cdots \bar{\theta}_{n-1}(x_p)-\exp(\omega) \right)_{n+1},
\end{equation}
where the subscript $n+1$ in the right hand side means taking the
degree $(n+1)$-part.

Let us consider the case $\mathcal{S}=\mathcal{S}_0$.
Then $\zeta=\prod_{i=1}^g [\alpha_i,\beta_i]$.
For $X,Y\in \widehat{T}_1$, by a direct computation, we have
\begin{equation}
\label{eq:5-2}
(1+X)(1+Y)(1+X)^{-1}(1+Y)^{-1}=1+\sum_{i,j\ge 0}(-1)^{i+j}[X,Y]X^iY^j.
\end{equation}
See Magnus-Karrass-Solitar \cite{MKS} \S5.5, (7a) for a similar formula.
Therefore in case $\mathcal{S}=\mathcal{S}_0$, (\ref{eq:5-1}) becomes
$$V_{n+1}=\left( \prod_{i=1}^gG \left( \bar{\theta}_{n-1}(\alpha_i)-1,
\bar{\theta}_{n-1}(\beta_i)-1 \right)
-\exp (\omega) \right)_{n+1},$$
where $G(X,Y)$ is the right hand side of (\ref{eq:5-2}).
From (\ref{eq:2-1}) and (\ref{eq:3-4}), we obtain the following
recursive formulas for $\ell^{\mathcal{S}_0}$:
$$\ell_n^{\mathcal{S}_0}(\alpha_i)=\frac{1}{n+1}\Phi(V_n^{B_i}),$$
$$\ell_n^{\mathcal{S}_0}(\beta_i)=\frac{-1}{n+1}\Phi(V_n^{A_i}).$$
In this way we can effectively compute the terms of $\ell^{\mathcal{S}_0}(\xi_i)$.
Here we give the first few terms of $\ell^{\mathcal{S}_0}$ for $g=1,2$.

\begin{exple}[the case of genus 1]
{\rm For simplicity, we write $\alpha_1=\alpha$, $\beta_1=\beta$
and $A_1=A$, $B_1=B$. Modulo $\widehat{T}_6$, we have}
\begin{eqnarray*}
\ell^{\mathcal{S}_0}(\alpha) &\equiv &
A+\frac{1}{2}[A,B]+\frac{1}{12}[B,[B,A]]
-\frac{1}{8}[A,[A,B]]+\frac{1}{24}[A,[A,[A,B]]] \\
& & -\frac{1}{720}[B,[B,[B,[B,A]]]]-\frac{1}{288}
[A,[A,[A,[A,B]]]]-\frac{1}{288}[A,[B,[B,[B,A]]]] \\
& & -\frac{1}{288}[B,[A,[A,[A,B]]]]
+\frac{1}{144}[[A,B],[B,[B,A]]]+\frac{1}{128}[[A,B],[A,[A,B]]];
\end{eqnarray*}
\begin{eqnarray*}
\ell^{\mathcal{S}_0}(\beta) &\equiv &
B-\frac{1}{2}[A,B]+\frac{1}{12}[A,[A,B]]
-\frac{1}{8}[B,[B,A]]+\frac{1}{24}[B,[B,[B,A]]] \\
& & -\frac{1}{720}[A,[A,[A,[A,B]]]]-\frac{1}{288}
[B,[B,[B,[B,A]]]]-\frac{1}{288}[B,[A,[A,[A,B]]]] \\
& & -\frac{1}{288}[A,[B,[B,[B,A]]]]
-\frac{1}{144}[[A,B],[A,[A,B]]]-\frac{1}{128}[[A,B],[B,[B,A]]].
\end{eqnarray*}

\begin{exple}[the case of genus 2]
{\rm Modulo $\widehat{T}_5$, we have}
\begin{eqnarray*}
\ell^{\mathcal{S}_0}(\alpha_1) &\equiv &
A_1+\frac{1}{2}[A_1,B_1] \\
& & +\frac{1}{12}[B_1,[B_1,A_1]]-\frac{1}{8}[A_1,[A_1,B_1]]
-\frac{1}{4}[A_1,[A_2,B_2]] \\
& & +\frac{1}{24}[A_1,[A_1,[A_1,B_1]]]-\frac{1}{10}[[A_1,B_1],[A_2,B_2]]
+\frac{1}{40}[A_1,[B_1,[A_2,B_2]]] \\
& & +\frac{1}{40}[A_1,[B_2,[A_2,B_2]]]+\frac{1}{40}[A_1,[A_1,[A_2,B_2]]]
+\frac{1}{40}[A_1,[A_2,[A_2,B_2]]];
\end{eqnarray*}

\end{exple}

\begin{eqnarray*}
\ell^{\mathcal{S}_0}(\beta_1) &\equiv &
B_1-\frac{1}{2}[A_1,B_1] \\
& & +\frac{1}{12}[A_1,[A_1,B_1]]-\frac{1}{8}[B_1,[B_1,A_1]]
-\frac{1}{4}[B_1,[A_2,B_2]] \\
& & +\frac{1}{24}[B_1,[B_1,[B_1,A_1]]]+\frac{1}{10}[[A_1,B_1],[A_2,B_2]]
+\frac{1}{40}[B_1,[A_1,[A_2,B_2]]] \\
& & +\frac{1}{40}[B_1,[A_2,[A_2,B_2]]]+\frac{1}{40}[B_1,[B_1,[A_2,B_2]]]
+\frac{1}{40}[B_1,[B_2,[A_2,B_2]]];
\end{eqnarray*}

\begin{eqnarray*}
\ell^{\mathcal{S}_0}(\alpha_2) &\equiv &
A_2+\frac{1}{2}[A_2,B_2] \\
& & +\frac{1}{12}[B_2,[B_2,A_2]]-\frac{1}{8}[A_2,[A_2,B_2]]
+\frac{1}{4}[A_2,[A_1,B_1]] \\
& & +\frac{1}{24}[A_2,[A_2,[A_2,B_2]]]-\frac{1}{10}[[A_1,B_1],[A_2,B_2]]
-\frac{1}{40}[A_2,[B_2,[A_1,B_1]]] \\
& & -\frac{1}{40}[A_2,[B_1,[A_1,B_1]]]-\frac{1}{40}[A_2,[A_2,[A_1,B_1]]]
-\frac{1}{40}[A_2,[A_1,[A_1,B_1]]];
\end{eqnarray*}

\begin{eqnarray*}
\ell^{\mathcal{S}_0}(\beta_2) &\equiv &
B_2-\frac{1}{2}[A_2,B_2] \\
& & +\frac{1}{12}[A_2,[A_2,B_2]]-\frac{1}{8}[B_2,[B_2,A_2]]
+\frac{1}{4}[B_2,[A_1,B_1]] \\
& & +\frac{1}{24}[B_2,[B_2,[B_2,A_2]]]+\frac{1}{10}[[A_1,B_1],[A_2,B_2]]
-\frac{1}{40}[B_2,[A_2,[A_1,B_1]]] \\
& & -\frac{1}{40}[B_2,[A_1,[A_1,B_1]]]-\frac{1}{40}[B_2,[B_2,[A_1,B_1]]]
-\frac{1}{40}[B_2,[B_1,[A_1,B_1]]].
\end{eqnarray*}

\end{exple}

\vspace{0.5cm}
\noindent \textbf{Acknowledgments.}
The author wishes to express his gratitude to
Alex Bene, who kindly suggested to him to extend the construction
for not necessarily symplectic generators,
Nariya Kawazumi for communicating to him a proof of a symmetry property of $\theta^0$,
and Robert Penner for warm comments to a rough draft of this paper. 
He also would like to thank
Shigeyuki Morita and Masatoshi Sato for valuable comments.

This research is supported by JSPS Research Fellowships
for Young Scientists (22$\cdot$4810).

\noindent \textsc{Yusuke Kuno\\
Department of Mathematics,\\
Graduate School of Science,\\
Hiroshima University,\\
1-3-1 Kagamiyama, Higashi-Hiroshima, Hiroshima 739-8526 JAPAN}

\noindent \texttt{E-mail address: kunotti@hiroshima-u.ac.jp}

\end{document}